\documentclass[sn-basic]{sn-jnl}% Basic Springer Nature Reference Style/Chemistry Reference Style
\jyear{2021}%

\theoremstyle{thmstyleone}%
\theoremstyle{thmstyletwo}%

\theoremstyle{thmstylethree}%

\newcommand{\pr}{\mbox{\sf P}}
\newcommand{\ex}{{\bf\sf E}}               %% expectation
\newcommand{\calg}{{\mathcal G}}

\newcommand{\al}{\alpha}                %%
\newcommand{\g}{\lambda}                %%
\newcommand{\ga}{\gamma}                %% abbreviated
\newtheorem{thm}{Theorem}
\newtheorem{lem}[thm]{Lemma}

\newtheorem{rem}[thm]{Remark}

\def\th{\theta}

\begin{document}

\title[Two-Dimensional Markov Chain]{On A Stein Method Based Approximation for A Two-Dimensional Markov Chain}


\author*[1,2]{\fnm{Yingdong} \sur{Lu}}\email{yingdong@us.ibm.com}



\affil*[1]{\orgdiv{Mathematical Sciences}, \orgname{IBM Research}, \orgaddress{\street{1101 Kitchawan Rd}, \city{Yorktown Heights}, \state{NY}, \postcode{10598}, \country{U.S.A.}}}

%\affil[2]{\orgdiv{Department}, \orgname{Organization}, \orgaddress{\street{Street}, \city{City}, \postcode{10587}, \state{State}, \country{Country}}}

%\affil[3]{\orgdiv{Department}, \orgname{Organization}, \orgaddress{\street{Street}, \city{City}, \postcode{610101}, \state{State}, \country{Country}}}


\abstract{We study an approximation method of stationary characters of a two-dimensional Markov chain  via the Stein method. For this purpose, innovative methods are developed to estimate the moments of the Markov chain, as well as the solution to the Poisson equation with a partial differential operator.}




\maketitle

\section{Introduction}
\label{sec:intro}

Computing expected function of ergodic Markov chains defined on multidimensional spaces that are not compact, with respect to their stationary distributions, is always a difficult problem. Analytic and algebraic methods are developed for some special cases, such as those Markov chains whose transition probabilities takes only a few possible values, see e.g.  \cite{fayolle1999random}. A popular approach of approximation is to calculated related quantities for a derived Markov chain on a finite state space, which can be calculated efficiently, see, e.g. for the studies in  \cite{MazalovGurtov2012,BSH2008} in this nature. 

We consider an approximation method of evaluating, through known differential equations techniques, the function against a diffusion process whose generator preserve the main characters of the generator of the Markov chain under a proper scaling. This method is elaborated through a two dimensional Markov chain motivated by a queueing application. Using the Stein method, coupled with with estimation by differential equation methods, we are able to quantify the error of this approximations through a comparison analysis of the generators. The Stein method~\cite{stein86} is a versatile technique in probability theory, rooted in the studies of the concentration of measures, such as the central limit theorems. Recent developments in \cite{doi:10.1287/moor.2013.0593, gurvich2014, braverman2017}, utilize the Stein method to estimate the stationary distribution of a Markov chain by that of a diffusion process, which is usually mathematical more tractable, by comparing the generators and the solution to the Poisson (Stein) equation. While our overall approach follow the same logic, the bounds on derivatives are different and innovative. 

The rest of paper will be organized as follows, in Sec. \ref{sec:models}, we provide detailed description of the Markov chain; In Sec. \ref{sec:generator_exp} we provide a generator expansion; and in Sec. \ref{sec:Stein_method}, the main results are discussed and proved. 

\section{A Two Dimensional Markov Chain}
\label{sec:models}

\subsection{Definition of the Markov Chain}
\label{sec:MC_defn}

The Markov chain model is motivated by the following queueing model. The job arrivals follow a Poisson process with rate $\g$, and service time is independently drawn after an exponential distribution with rate $\mu$. Meanwhile, a stream of servers arrive, also following an independent Poisson process with rate $\gamma$. When a job arrives, it will be served immediately if there are any idle servers, otherwise it will join a single queue in front of all the servers. Whenever there is a server becomes available, due to either the departure of a job or the arrival of a new server, the jobs in queue will be served in a first-come-first-serve (FCFS) fashion. Meanwhile, each server that becomes idle will start an independent departure clock, which follows an exponential distribution with rate $\nu$, the server will depart if the clock expires before it takes on a job. In other words, a server will leave the system after staying idle for a random time period (exponential with rate $\nu$).

The system can be characterized with a two-dimensional continuous time Markov chain(CTMC). The state space is ${\mathbb Z}_+ \times {\mathbb Z}_+ $ with ${\mathbb Z}_+ $ denoting the set of all nonnegative integers. A state $(i, j)\in {\mathbb Z}_+ \times {\mathbb Z}_+ $ consists of the number of jobs in the system, $i$, and the number of servers in the system, $j$. The transition rates are in the following form,
\begin{align}
\label{eqn;rates}
\left\{ \begin{array}{cc} (i,j) \rightarrow (i+1,j ) & \g \\ (i,j) \rightarrow (i-1,j ) & (i \wedge j) \mu \\ (i,j)\rightarrow (i,j+1 ) & \gamma \\(i,j) \rightarrow (i,j-1 ) & (j-i)^+\nu \end{array}\right.,
\end{align}
where $x\wedge y := \min\{x, y\}$ and $(x-y)^+:=\max\{x-y, 0\}$.  From \eqref{eqn;rates}, we can write the transition rate matrix (which is of  infinite dimension) in the following form. 

\begin{itemize}
\item
For the state $(i,0)$, there are only two events can happen, the arrival of a job and the arrival of a server,  and with rate $\g$ and $\ga$ respectively. Hence, $q_{(i,0) \rightarrow (i+1,0)}=\frac{\g}{\g+\gamma}$,$q_{(i,0) \rightarrow (i,1)}=\frac{\gamma}{\g+\gamma}$.
\item
For any states in the form of $(0, j)$, for $j\ge 1$, three events can happen, job arrival, server arrival and departure. Hence, $q_{(0,j) \rightarrow (1,j)}=\frac{\g}{\g+\gamma+j \nu}$, $q_{(0,j) \rightarrow (0,j+1)}=\frac{\gamma}{\g+\gamma+j \nu}$, and $q_{(0,j) \rightarrow (0,j-1)}=\frac{j\nu}{\g+\gamma+j \nu}$.
\item
For any states in the form of $(i,j)$, for $i, j\ge 1$, the transition probabilities are  $q_{(i,j) \rightarrow (i+1,j)}=\frac{\g}{\g+(i\wedge j) \mu+\gamma+(j-i)^+\nu}$, $q_{(i,j) \rightarrow (i-1,j)}=\frac{(i\wedge j) \mu}{\g+(i\wedge j) \mu+\gamma+(j-i)^+\nu}$, $q_{(i,j) \rightarrow (i,j+1)}=\frac{\gamma}{\g+(i\wedge j) \mu+\gamma+(j-i)^+\nu}$, and $q_{(i,j) \rightarrow (i,j-1)}=\frac{(j-i)^+\nu}{\g+(i\wedge j) \mu+\gamma+(j-i)^+\nu}$. Note that $q_{(i,j) \rightarrow (i,j-1)}$ could be zero when $j\le i$. 
\end{itemize}
Let us denote the Markov chain $(X(t), Y(t))$, and its generator $\calg_0$. For any bounded function $f: {\mathbb Z}_+^2 \rightarrow {\mathbb R}$,
\begin{align*}
  \calg_0 f(i, j) = & \g [f(i+1,j) -f(i,j)] + (i\wedge j) \mu [f(i-1, j) - f(i,j)] \\ &+ \gamma [f(i, j+1)-f(i,j)] + (j-i)^+ \nu [f(i, j-1)-f(i,j)].
\end{align*}
\begin{lem}
\label{lem:stationary_finiteness}
$(X(t), Y(t))$ has a stationary distribution, and more importantly, the stationary distribution has finite third moment. 
\end{lem}
\begin{proof}
Apply the generator $\calg_0$ to function $f(x,y)=x^4+y^4$, we have, 
\begin{align*}
\calg_0 f(x, y) = & \g[(x+1)^4-x^4]+(x\wedge y)[(x-1)^4-x^4]+ \gamma[(y+1)^4-y^4] \\&+ (y-x)^+[(y-1)^4-y^4] \\ = & \g (4x^3+6x^2 + 4x+1) + (x\wedge y)   (-4x^3+6x^2 - 4x+1) \\  &+ \gamma(4y^3+6y^2+4y+1) + (y-x)^+(-4y^3+6y^2-4y+1).
\end{align*}
Therefore, it is easy to see that there exist a $(x_0, y_0)$, and $C>0$, such that when $(x, y) \ge (x_0, y_0)$, $\calg_0 f(x, y) \le -C (x^3+y^3)$. Thus, $\calg_0 f(x, y) \le -C (x^3+y^3) +C'{\bf 1}\{(x, y) \ge (x_0, y_0)\} $, and by Theorem 4.2 in \cite{meyn_tweedie_1993}, the Markov chain $(X(t), Y(t))$ has a stationary distribution, and it has finite third moment.  
\end{proof}

\subsection{Centering, Scaling and the Scaled Processes}
\label{sec: scaling}

To facilitate our analysis, we will consider the following ''centerred'' and ''scaled'' Markov chain through translation and scaling. Any stationary function calculations for the original process can be readily transformed to the ones for the centered and scaled Markov chain. 

\subsubsection{Centering}

First, to find the equilibrium point $(x(\infty),y(\infty))$, the pair that represents the equilibrium number of the jobs and servers, consider the following system of flow balance equations,
\begin{align*}
  \g = [x(\infty)\wedge y(\infty)]\mu, \qquad \gamma = [y(\infty) - x(\infty)]^+ \nu.
\end{align*}
Since, under our assumptions, $\gamma>0$ and $\nu>0$, the second equation implies  $y(\infty)-x(\infty)>0$ and $y(\infty)-x(\infty)= \frac{\ga}{\nu}$. 
Combined with the first equation, we have,
\begin{align}
\label{eqn:centering}
  x(\infty) = \frac{\g}{\mu}, \quad y(\infty) =  \frac{\g}{\mu}+  \frac{\gamma}{\nu}.
\end{align}

\subsubsection{Scaling}

The solutions in \eqref{eqn:centering} indicates that given the parameter $(\g, \mu, \gamma, \nu)$, the queue length and number of servers will be in essence approaching the above equilibrium point. Consider a sequence of systems, indexed by $n$, such that, 
\begin{align}
\label{eqn:scaling}
\g_n = n \mu, \mu_n = \mu, \gamma_n= \kappa n^{\al} \nu, \nu_n =\nu, 
\end{align}
for some positive real number $\mu$, $\kappa$ and $\al$. Hence, the equilibrium states are $x_n(\infty) = n$ and $y_n(\infty) = n + \kappa n^\al$.  Note that $\al=\frac12$ represents the famed Halfin-Whitt scaling~\cite{RePEc:inm:oropre:v:29:y:1981:i:3:p:567-588}. 

\subsubsection{Stationary Function Calculations}

For approximating the stationary performance, it is more convenient to consider the "centered" and "scaled" version of the Markov chain $(X(t), Y(t))$. 
Define,
\begin{align*}
{\tilde  X}^n(t)= \delta[X^n(t)-n], {\tilde Y}^n(t)= \eta[Y^n(t)-n-\kappa n^\al]
\end{align*}
for some scaling factors $\delta$ and $\eta$ that tends to zero as $n$ grows. For example, in the case of Halfin-Whitt scaling ($\al=\frac12$), $\delta $ and $\eta$ can also be choose to be $\frac12$. The generator $\calg_n$ for $({\bar X^n}, {\bar Y}^n)$ can be written in the following form, for any bounded smooth function  $u: {\mathbb R}^2 \rightarrow {\mathbb R}$, with $x=\delta (i - n)$, $y= \eta (j-n-\kappa n^\al)$ (hence, $i=\frac{x}{\delta}+n$, and $j= \frac{y}{\eta}+n+\kappa n^\al$),
\begin{align}
\nonumber
  \calg_n u(x, y) = & \g_n [u(x+\delta,y) -u(x,y)]  + b^1_n(x,y) \mu [u(x-\delta, y) - u(x,y)] \\ &+ \gamma_n [u(x, y+\delta)-u(x,y)] +b^2_n(x,y)\nu [u(x, y-\delta)-u(x,y)], \label{eqn:generator_n}
\end{align}
with 
\begin{align*}
b^1_n(x,y)&:= \left[\left(\frac{x}{\delta}+n\right)\wedge \left(\frac{y}{\eta}+ (n+\kappa n^\al)\right)\right], \\b^2_n(x,y)&:= \left(\frac{y}{\eta}-\frac{x}{\delta}+\kappa n^\al \right)^+. 
\end{align*}

Let function $h(x,y)$ be the quantity of interest, for example, in the motivating queueing system, it can represent the performance of the system that depends on both the number of jobs and the number of servers. The stationary function calculation takes the form of  $\ex[h({\tilde  X}^n(\infty), {\tilde  Y}^n(\infty))]$ with $({\tilde  X}^n(\infty), {\tilde  Y}^n(\infty))$ denoting the stationary distribution of the process $({\tilde  X}^n(t), {\tilde  Y}^n(t))$.

%\section{Process Approximations}
%\label{sec:process_approx}

%In this paper, we derive finite version of the approximation, instead of fluid and diffusion approximations. First, we will derive different levels of approximation via generator expansion, these can be considered to be counterparts of fluid and diffusion approximations. Then, we derive the approximation for the stationary distributions. For both purpose, we need to study the generator very carefully. \cite{doi:10.1137/1119062}.

\section{Generator Expansion}
\label{sec:generator_exp}

For any $(x,y)\in {\mathbb R}^2$, the Taylor expansion of the function $u(x, y)$ at $(x,y)$ will help us in expanding the generator $\calg_n$ in \eqref{eqn:generator_n}, and identifying the approximating diffusion process. More specifically, we have,   
\begin{align*}
\g_n  [u(x+\delta,y) -u(x,y)]  =&  \g_n \delta u_x(x, y) + \frac{\g _n \delta^2}{2}u_{xx}(\xi_1, y) \\= &\g_n  \delta u_x(x, y) + \frac{\g_n  \delta^2}{2} u_{xx}(x, y) \\ &+ \frac{\g _n \delta^2}{2}[u_{xx}(\xi_1, y)-u_{xx}(x, y)],
\end{align*}
with some $\xi_1\in[x, x+\delta]$, and 
\begin{align*}
u_x(x,y)&:= \frac{\partial}{\partial x}u(x,y), & u_y(x,y)&:= \frac{\partial}{\partial y}u(x,y), \\ u_{xx}(x,y)&:= \frac{\partial^2}{\partial x^2}u(x,y), &u_{yy}(x,y)&:= \frac{\partial^2}{\partial y^2}u(x,y).
\end{align*}
Next, 
\begin{flalign*}
& b^1_n(x,y)\mu [u(x-\delta, y) - f(x,y)]\\= & b^1_n(x,y)\mu \left[ -\delta u_x(x, y) + \frac{ \delta^2}{2} u_{xx}(\xi_2, y) \right]\\= & - b^1_n(x,y)\mu \delta u_x(x, y) + \frac{b^1_n(x,y) \mu \delta^2}{2} u_{xx}(\xi_2, y) 
 %+ \frac{b_1(x,y)\mu \delta^2}{2}[u_{xx}(\xi_2, y)-u_{xx}(x, y)],
\end{flalign*}
with some $\xi_2 \in [x-\delta, x]$. Similarly, we have, 
\begin{align*}
 \gamma_n [u(x, y+\eta)-u(x,y)] & = \gamma_n \eta u_y(x, y)
 + \gamma_n \eta [u_y(x, \xi_3)-u_y(x, y)],
 %+ \frac{\gamma \delta^2}{2} u_{yy}(x, y) + \frac{\gamma \delta^2}{2}[u_{yy}(x, \xi_3)-f_{yy}(x, y)],
\end{align*}
with some $\xi_3\in [y, y+\eta]$. And
\begin{align*}
 & b^2_n(x,y) [u(x-\eta, y)-u(x,y)]  \\ =&  - b^2_n(x,y) \nu  \eta u_y(x, y)-  b^2_n(x,y) \nu  \eta [u_y(x, \xi_4)-u_y(x, y)],
 %+ \frac{ b_2(x,y) \nu \delta^2}{2} u_{yy}(x, y) + \frac{b_2(x,y) \nu \delta^2}{2}[u_{yy}(x, \xi_4)-y_{yy}(x, y)].
\end{align*}
with some $\xi_4\in[y-\eta, y]$.

Now, let us examine the behavior of each term we obtained through above expansion, and explain the rational of the selection of the generator below.  First, let us look at the terms with the first order terms in the generator. 

\subsubsection{Terms of the first order ($u_x$ and $u_y$)}

The coefficient for $u_x$ is $\delta[ \g_n -b^1_n(x,y)\mu]$, which equals to, 
 \begin{align*}
&\delta\g_n-\left[(x+n\delta)\wedge \left(\frac{y\delta}{\eta}+ (n+\kappa n^\al)\delta\right)\right]\mu  \\=& (n\delta) \mu - \left[(x+n\delta)\wedge \left(\frac{y\delta}{\eta}+ (n+\kappa n^\al)\delta\right)\right]\mu \\ =& -\left[x\wedge \left(\frac{y\delta}{\eta}+\delta\kappa n^\al\right)\right]\mu. 
\end{align*}
Meanwhile, the coefficient for $u_y$ is 
\begin{align*}
\eta[ \gamma -b^2_n(x,y)\nu]&= \eta \kappa n^\al \nu -  \left(y-\frac{x\eta}{\delta}+ \eta\kappa n^\al \right)^+\nu.
\end{align*}
Observe that, both terms of $\frac{\eta}{\delta}$ and $\frac{\delta}{\eta}$ are present.

%Of course,  $n\delta=n^{1/2}$ , so the above term will become $x\wedge (y+\kappa)$, hence it is an $O(1)$ term when $n$ grows. This is consist with the magnitude also observed in the Halfin-Whitt regime heavy traffic analysis for queueing. 
%\subsubsection{$\delta[ \gamma -b_2(x,y)\nu]$}

%\begin{rem}
%Whatever scaling parameters $\delta$ and $\th$ are set, it looks like that the first order terms of the diffusion processes will be just the above two terms. In the approximation analysis, we don't need to worry about these terms. 
%\end{rem}

\subsubsection{Terms of the second order $u_{xx}$ and $u_{yy}$}
%\subsubsection{$\frac{\delta^2}{2}(\g+b_1(x,y) \mu)$}
The coefficient for $u_{xx}$ is $\frac{\delta^2}{2}(\g_n+b^1_n(x,y) \mu)$. 
Apply the above scaling, we can see that, 
\begin{align*}
\frac{\delta^2}{2}[\g_n+b^1_n(x,y) \mu] &= \frac{\delta^2}{2}\left[n+ (x+ n\delta) \wedge  \left(\frac{y\delta}{\eta}+ (n+\kappa n^\al)\delta\right)\right]\mu.
\end{align*}
It can be seen that the first term is $\delta^2n/2$, meanwhile, 
\begin{align*}
\frac{\delta^2}{2}[(x+ n\delta) \wedge (y+ n\delta+\kappa n^\al\delta)]\mu,
\end{align*}
will be of the order of $O(n\delta^3)$, and can be treated as an error term.  To bound this error term, we need the moment bound ( i.e. the first moment bound), more specifically, we need to show that the Markov chain has 
finite first moment.

%\subsubsection{$\frac{\delta^2}{2} [\gamma + b_2(x,y)\nu]$}
The coefficient for $u_{yy}$ is $\frac{\eta^2}{2} [\gamma_n + b^2_n(x,y)\nu]$. 
\begin{align*}
\frac{\eta^2}{2} [\gamma_n + b^2_n(x,y)\nu]
&= \frac{\eta^2}{2}\left[\kappa n^\al + \left(y-\frac{x\eta}{\delta}+ \eta\kappa n^\al \right)^+\right]\nu.
\end{align*}
%This one looks though that can be ignored completely. The indication is that for the $y$ direction, we have a PDE instead of a SDE. 
Again, to ensure that this error term is small, we need an estimate of the first order quantities.

%From the expression of the first two order terms, it looks like that we should have make $\delta$ to be $\sqrt{n}$, so that the first order, which is essentially the difference between job arrivals and number of servers, as well as the second order term, should be finite. Note that the coefficients of the diffusion process could still include $n$. 

Therefore, \eqref{eqn:generator_n} can be written as, 
\begin{align}
\label{eqn:error_terms_intro}
\calg_n u(x, y) = & \calg u(x, y)+ E_1+ E_2 + E_3+E_4,
\end{align}
where $\calg$ is given by, 
\begin{align*}
\calg u(x, y) = & \delta[\g_n -b^1_n(x,y)\mu]\frac{\partial}{\partial x}u(x,y) + \delta[ \gamma_n -b^2_n(x,y)\nu] \frac{\partial}{\partial x}u(x,y)  %+\frac{\delta^2}{2} (\g + b_1(x,y)\mu)
+\frac12\frac{\partial^2}{\partial x^2}  u(x,y),
\end{align*}
representing the generator of a diffusion process defined as
\begin{align}
\left\{ \begin{array}{ccc} dX_t &=& \delta[\g_n -b^1_n(X_t,Y_t)\mu] dX_t +\frac12 dW_t,\\ dY_t &=&\delta[ \gamma -b^2_n(X_t, Y_t)\nu] dY_t,
\end{array} \right. \label{eqn:diffusion}
\end{align}
with $W_t$ being a standard Brownian motion. 
$E_i, i=1,\ldots,4$ are error terms that will be estimated below, \begin{align*}
E_1&= \frac{\g \delta^2}{2}[u_{xx}(\xi_1, y)-u_{xx}(x, y)], & E_2 &= \frac{b^1_n(x,y)\mu \delta^2}{2}u_{xx}(\xi_2, y),\\ E_3 &= \gamma_n \delta [u_y(x, \xi_3)-u_y(x, y)], & E_4 &=-  b^2_n(x,y) \nu  \delta [u_y(x, \xi_4)-u_y(x, y)].
\end{align*}

%It is then understandable to approximate the stationary performance of the Markov chain that that of the diffusion process. 
The diffusion process \eqref{eqn:diffusion} can be viewed as a stochastic Hamiltonian system, a general overview can be founded in e.g. \cite{soizebook}, and a detailed analysis on its stationary behavior are presented in \cite{Talay2002StochasticHS}. In \cite{Talay2002StochasticHS}, numerical methods are also discussed in the cases that exact form of the stationary distribution can not be obtained. 

\section{Stein Method for Error Estimation}
\label{sec:Stein_method}

In this section, we present a detailed analysis on the approximation error via the Stein method. Especially, we will quantity the four error terms identified in the above analysis, which guide the derivations of the moment and derivative bound in the sections below.

%\subsection{Stationary Distributions}
%\label{eqn:stationary_dist}

%The worst case, we can use methods developed in the literature to have a computation of the stationary. We need to point out that this will not be a topic that will be discussed in this paper. 

\subsection{Main Results}
\begin{lem}
\label{lem:vanishing}
Let $f(x,y)$ be a function such that $|f(x,y)|\le C(1+|x|^3+|y|^3)$ for some $C>0$, then $\ex[\calg_n f({\tilde X}^n(\infty) , {\tilde X}^n(\infty) )]=0$. 
\end{lem}
\begin{proof}
As indicated in \cite{gurvich2014, braverman2017}, it suffices to know that $(X(t),Y(t))$ is positive recurrent and the stationary distribution has finite third moment,  and that is established in Lemma \ref{lem:stationary_finiteness}.
\end{proof}
Meanwhile, the following order estimations of the error terms will be proved in Sec. \ref{sec:error},
\begin{lem}
\label{lem:error}
\begin{align*}
\ex[E_i] &= O(n\delta^3),i=1,2, \quad  \ex[E_i] = O(n^\al \eta^2), i=3,4.
\end{align*}
\end{lem}
Thus,
\begin{thm}
\label{thm:main}
For a performance metric function satisfies that $(|x|^p+|y|^p) h(x,y)$ is integrable for $p \le 5$, we can conclude that, there exists a constant $C>0$, such that $\vert\ex h({\tilde X}^n(\infty), {\tilde Y}^n(\infty))-\ex h(X_\infty, Y_\infty)\vert\le C(n\delta^3+n^\al \eta^2)$, with $({\tilde X}^n(\infty), {\tilde Y}^n(\infty))$ and $(X_\infty, Y_\infty)$ represent the stationary distribution of the centered and scaled Markov chain $({\bar X}^n(t), {\bar Y}^n(t))$, and the diffusion process defined in \eqref{eqn:diffusion}, respectively. 
\end{thm}
\begin{proof}
Recall that for each $n$, the scaled ans centered process is $({\tilde X}^n(t), {\tilde Y}^n(t))$ with generator $\calg_n$. The goal is to estimates the average difference of performance, 
\begin{align}
\label{eqn:stein_subject}
\vert \ex[h({\tilde X}^n(\infty), {\tilde Y}^n(\infty))]-\ex[h(X_\infty, Y_\infty)]\vert.
\end{align}
with $(({\bar X}(t), {\bar Y}(t))$ denotes the approximating process, and $h(x,y)$ a general performance metric function. This function can be very general, could cover probability based performance as seen in many applications. Let us denote $u^h$ be the solution to the Stein equation, 
\begin{align}
	\label{eqn:stein}
	\calg u = h(x,y)-\ex[h(X_\infty, Y_\infty)]. 
\end{align}
Apply the expectation with the stationary distribution for the $n$-th system, we have, 
\begin{align*}
\ex[ \calg u^h({\tilde X}^n(\infty), {\tilde Y}^n(\infty))]= \ex[h({\tilde X}^n(\infty), {\tilde Y}^n(\infty))]-\ex[h(X_\infty, Y_\infty)].
\end{align*}
This can be written as,
\begin{align*}
&\ex[ (\calg-\calg_n)u^h({\tilde X}^n(\infty), {\tilde Y}^n(\infty))]+\ex[ \calg_n u^h({\tilde X}^n(\infty), {\tilde Y}^n(\infty))] \\= &\ex[h({\tilde X}^n(\infty), {\tilde Y}^n(\infty))]-\ex[h(X(\infty), Y(\infty))].
\end{align*}
Basic property of the generator, Lemma \ref{lem:vanishing},  implies that the term $\ex[ \calg_n u^h({\tilde X}^n(\infty), {\tilde Y}^n(\infty))]$ vanishes. Hence, we have, 
\begin{align*}
 \ex[h({\tilde X}^n(\infty), {\tilde Y}^n(\infty))]-\ex[h(X_\infty, Y_\infty)]=\ex[ (\calg-\calg_n)u^h(X^n(\infty), Y^n(\infty)].
\end{align*}
Thus, we only need to estimate the right hand side. The expression in \eqref{eqn:error_terms_intro} affirms that we only need to estimate $\ex[E_1]+ \ex[E_2]+\ex[E_3]+\ex[E_4]$, which is provided in the Lemma \ref{lem:error}. %corollary  \ref{cor:unify}. 
\end{proof}

\begin{rem}
When $\al=\frac12$, $\delta=\eta=n^{-1/2}$, the result in Theorem \ref{thm:main} is consist with the Halfin-Whitt type of results that are  well-known in the queueing literature. In general, we can see that the approximation depends on the rate of server arrival. 
\end{rem}
\subsection{Error Estimates}
\label{sec:error}

In this section, we will provide the basic estimation of the error terms. This consists of two parts, In Sec. \ref{sec:moment_bound}, we will discuss the bounds related to the first and second moments of the variable $({\tilde X}^n(\infty), {\tilde Y}^n(\infty))$; in 
Sec. \ref{sec:gradient_bound}, we present arguments for bounding the derivatives of the Stein equation \eqref{eqn:stein}. 

\subsubsection{Moment Bounds}
\label{sec:moment_bound}

Recall the generator  for the Markov chain indexed by $n$,  $\calg_n$,
\begin{align*}
	\calg_n u(x, y) = &  \g_n  [u(x+\delta,y) -u(x,y)] + b^1_n(x,y)\mu [u(x-\delta, y) - u(x,y)] \\ &+ \gamma_n [u(x, y+\delta)-u(x,y)] + b^2_n(x,y) \nu [u(x, y-\delta)-u(x,y)].
\end{align*}
%\begin{align*}
%  \calg_n f(x, y) = & \g [f(x+\delta,y) -f(x,y)] \\ &+ \frac{1}{\delta}\left[\left(x+n\delta\right)\wedge \left(y+ (n+\kappa n^\al)\delta\right)\right] \mu [f(x-\delta, y) - f(x,y)] \\ &+ \gamma [f(x, y+\delta)-f(x,y)] \\ &+ \frac{1}{\delta}\left(y-x+\delta\kappa n^\al \right)^+ \nu [f(x, y-\delta)-f(x,y)].
%\end{align*}
\begin{lem}
	\label{lem:b1_mean} $\ex[b^1_n({\tilde X}^n(\infty), {\tilde Y}^n(\infty))]=\frac{\g_n }{\mu}$.
\end{lem}
\begin{proof}
Let $u(x,y)=x$, Lemma \ref{lem:vanishing} implies, 
\begin{align}
\label{eqn:x_function}
\ex\left[\left({\tilde X}^n(\infty)+n\delta\right)\wedge \left(\frac{{\tilde Y}^n(\infty)\delta}{\eta}+ (n+\kappa n^\al)\delta\right)\right] =\frac{\g_n }{\mu}.
\end{align}
That gives the desired expression of  $\ex[b^1_n({\tilde X}^n(\infty), {\tilde Y}^n(\infty))]$.
\end{proof}
\begin{lem}
	\label{lem:b2_mean}
	 $\ex[b^2_n({\tilde X}^n(\infty), {\tilde Y}^n(\infty))]=\frac{\gamma}{\nu}$
\end{lem}
\begin{proof} Let $u(x,y)=y$,  Lemma \ref{lem:vanishing} implies, 
\begin{align}
\label{eqn:y_function}
\ex\left[\left(\frac{{\tilde Y}^n(\infty)}{\eta}-\frac{{\tilde X}^n(\infty)}{\delta}+\delta\kappa n^\al \right)^+\right] = \frac{\gamma}{\nu}.
\end{align}
That gives the desired expression of $\ex[b^2_n({\tilde X}^n(\infty), {\tilde Y}^n(\infty))]$.
\end{proof}
Furthermore, 
\begin{lem}
	\begin{align}
		\label{eqn:Y_mean}
		\ex[ {\tilde Y}^n(\infty)] \le \left(\frac{\g_n  \delta}{\mu}+ \frac{\gamma\delta}{\nu}\right) \eta.%- [(\th-\delta)\wedge 0]\kappa n^\al .
	\end{align}
\end{lem}
\begin{proof}
It is easy to verify that the following inequality holds (due to an elementary inequality $x\wedge y + (y+a -x)^+\ge y+(a\wedge 0)$),
\begin{align*}
&\left[\left({\tilde X}^n(\infty)+n\delta\right)\wedge \left(\frac{{\tilde Y}^n(\infty)\delta}{\eta}+ (n+\kappa n^\al)\delta\right)\right]+\left(\frac{{\tilde Y}^n(\infty)}{\eta}-\frac{{\tilde X}^n(\infty)}{\delta}+\delta\kappa n^\al \right)^+\\ \ge &  \frac{{\tilde Y}^n(\infty)}{\eta}.%+ [(\th-\delta)\wedge 0]\kappa n^\al
\end{align*}
Therefore, \eqref{eqn:Y_mean} follows immediately from \eqref{eqn:x_function} and \eqref{eqn:y_function}. 
\end{proof}

\begin{lem}
\label{lem:moment_ind}
\begin{align*}
 \ex\left[({\tilde X}^n(\infty)+n\delta){\bf 1}\left\{ \frac{{\tilde X}^n(\infty)}{\delta}+ n \le \frac{{\tilde Y}^n(\infty)}{\eta} +(n+\kappa n^\al)\right\}\right]\le \frac{\delta \g_n }{\mu}.
\end{align*}
\end{lem}

\begin{proof}[Proof of Lemma \ref{lem:moment_ind} ]
Let $u(x,y)=x{\bf 1}\{ \frac{x}{\delta}+ n \le \frac{y}{\eta} +(n+\kappa n^\al)\}$,   to apply Lemma \ref{lem:vanishing}, we need the following calculations.
\begin{align*}
u(x+\delta, y) -u(x,y) =& (x+\delta) {\bf 1}\{ \frac{x}{\delta}+ (n+1) \le \frac{y}{\eta}+(n+\kappa n^\al)\}\\ &- x{\bf 1}\{ \frac{x}{\delta}+ n \le \frac{y}{\eta} +(n+\kappa n^\al)\}\\  = & \delta {\bf 1}\{ \frac{x}{\delta}+ (n+1) \le \frac{y}{\eta} +(n+\kappa n^\al)\}\\ &+ x{\bf 1}\{ \frac{y}{\eta} +(n+\kappa n^\al) -\delta < \frac{x}{\delta}+ n \le \frac{y}{\eta}+(n+\kappa n^\al)\}.
\end{align*}
\begin{align*}
u(x-\delta, y) -u(x,y) = &(x-\delta) {\bf 1}\{ \frac{x}{\delta}+ (n-1) \le \frac{y}{\eta} +(n+\kappa n^\al)\}\\ &- x{\bf 1}\{ \frac{x}{\delta}+ n \le \frac{y}{\eta} +(n+\kappa n^\al)\}\\ =& -\delta {\bf 1}\{ \frac{x}{\delta}+ n \le \frac{y}{\eta} +(n+\kappa n^\al)\} \\ &+ (x-\delta) {\bf 1}\{ \frac{y}{\eta} +(n+\kappa n^\al) < \frac{x}{\delta}+ n \le \frac{y}{\eta} +(n+\kappa n^\al)+1\}.
\end{align*}
\begin{align*}
u(x, y+\delta) -u(x,y) =& x{\bf 1}\{ \frac{x}{\delta}+ (n-1) \le \frac{y}{\eta} +(n+\kappa n^\al)\}\\ &- x{\bf 1}\{ \frac{x}{\delta}+ n \le \frac{y}{\eta} +(n+\kappa n^\al)\}\\ =& x{\bf 1}\{ \frac{y}{\eta} +(n+\kappa n^\al) < \frac{x}{\delta}+ n\le \frac{y}{\eta}+(n+\kappa n^\al)+1\}.
\end{align*}
\begin{align*}
u(x, y-\delta) -u(x,y) =& x{\bf 1}\{ \frac{x}{\delta}+ (n+1) \le \frac{y}{\eta}  +(n+\kappa n^\al)\}\\&- x{\bf 1}\{ \frac{x}{\delta}+ n \le \frac{y}{\eta}  +(n+\kappa n^\al)\}\\ =& -x{\bf 1}\{\frac{y}{\eta}  +(n+\kappa n^\al) -1 < \frac{x}{\delta}+ n \le \frac{y}{\eta}  +(n+\kappa n^\al)\}.
\end{align*}
Therefore, on the set $( -\infty, \frac{{\tilde Y}^n(\infty)\delta}{\eta} +(n+\kappa n^\al)\delta-\delta]$, we have, $\g \delta -b^1_n\delta$; on the set $(  \frac{{\tilde Y}^n(\infty)\delta}{\eta} +(n+\kappa n^\al)\delta-\delta,  \frac{{\tilde Y}^n(\infty)\delta}{\eta} +(n+\kappa n^\al)\delta]$, we have, $\g x - b^1_n \delta -b^2_n x$, $( \frac{{\tilde Y}^n(\infty)\delta}{\eta} +(n+\kappa n^\al)\delta, \frac{{\tilde Y}^n(\infty)\delta}{\eta} +(n+\kappa n^\al)\delta+\delta]$, we have, $(x-\delta) b^1_n+ \nu x$, and beyond $ \frac{{\tilde Y}^n(\infty)\delta}{\eta} +(n+\kappa n^\al)\delta+\delta$, the value is zero.
Thus, 
\begin{align*}
&\delta \g_n  \pr[{\tilde X}^n(\infty)+ n\delta \le \frac{{\tilde Y}^n(\infty)\delta}{\eta} +(n+\kappa n^\al)\delta] \\ & - \mu\ex[({\tilde X}^n(\infty)+n\delta){\bf 1}\{ {\tilde X}^n(\infty)+ n\delta \le \frac{{\tilde Y}^n(\infty)\delta}{\eta} +(n+\kappa n^\al)\delta\}]=0.
\end{align*}
Thus, we have, 
\begin{align*}
 \ex[({\tilde X}^n(\infty)+n\delta){\bf 1}\{ {\tilde X}^n(\infty)+ n\delta \le\frac{{\tilde Y}^n(\infty)\delta}{\eta} +(n+\kappa n^\al)\delta\}]\le \frac{\delta \g}{\mu}.
\end{align*}
\end{proof}

Moreover, we can obtain the second moments of $b^1_n$ and $b^2_n$.
\begin{lem}
	\label{lem:2nd_moment_b_1}
	\begin{align*}
		\ex[(b^1_n)^2]\le \frac12\ex[\g (2\delta {\tilde X}^n(\infty) +\delta^2)]+(2n+1)\delta \ex\left[\left({\tilde X}^n(\infty)+n\delta\right)\wedge \left(\frac{{\tilde Y}^n(\infty)\delta}{\eta}+ (n+\kappa n^\al)\delta\right)\right] 
	\end{align*}
\end{lem}
%\begin{lem}
%\label{lem:2nd_moment_b_1}
%\begin{align*}\ex[b_1^2]\le \frac12\ex[\g (2\delta {\tilde X}^n(\infty) +\delta^2)]+(2n+1)\delta \ex\left[\left({\tilde X}^n(\infty)+n\delta\right)\wedge \left({\tilde Y}^n(\infty)+ (n+\kappa n^\al)\delta\right)\right] \end{align*}
%\end{lem}
\begin{proof}[Proof of Lemma \ref{lem:2nd_moment_b_1}]
Set $u(x,y)=x^2$, we have, 
\begin{align*}
\ex\left[\g_n  (2\delta {\tilde X}^n(\infty) +\delta^2) -\frac{1}{\delta} \left[\left({\tilde X}^n(\infty)+n\delta\right)\wedge \left(\frac{{\tilde Y}^n(\infty)\delta}{\eta}+ (n+\kappa n^\al)\delta\right)\right](2\delta {\tilde X}^n(\infty) -\delta^2)\right]=0.
\end{align*}
Thus,
\begin{align*}
\ex[\g_n  (2\delta {\tilde X}^n(\infty) +\delta^2)] =&\ex\left[\left[\left({\tilde X}^n(\infty)+n\delta\right)\wedge \left(\frac{{\tilde Y}^n(\infty)\delta}{\eta}+ (n+\kappa n^\al)\delta\right)\right](2 {\tilde X}^n(\infty) -\delta)\right]\\ =& \ex\left[\left[\left({\tilde X}^n(\infty)+n\delta\right)\wedge \left(\frac{{\tilde Y}^n(\infty)\delta}{\eta}+ (n+\kappa n^\al)\delta\right)\right](2 {\tilde X}^n(\infty) +2n\delta-(2n+1)\delta)\right].
\end{align*}
Hence, 
\begin{align*}
&\ex[\g_n  (2\delta {\tilde X}^n(\infty) +\delta^2)]+(2n+1)\delta \ex\left[\left({\tilde X}^n(\infty)+n\delta\right)\wedge \left(\frac{{\tilde Y}^n(\infty)\delta}{\eta}+ (n+\kappa n^\al)\delta\right)\right]\\=& \ex\left[\left[\left({\tilde X}^n(\infty)+n\delta\right)\wedge \left(\frac{{\tilde Y}^n(\infty)\delta}{\eta}+ (n+\kappa n^\al)\delta\right)\right](2 {\tilde X}^n(\infty)+2n\delta)\right]
\\ \ge & 2\ex\left[\left[\left({\tilde X}^n(\infty)+n\delta\right)\wedge \left(\frac{{\tilde Y}^n(\infty)\delta}{\eta}+ (n+\kappa n^\al)\delta\right)\right]\right]^2.
\end{align*}
This produces an upper bound on the second moment of $b^1_n$.
\end{proof}
Similarly,
\begin{lem}
\label{lem:2nd_moment_b_2}
\begin{align*}	\ex[(b^2_n)^2]\le &(\delta \gamma + \g_n \delta)\ex[({\tilde Y}^n(\infty))] - \ex[	\ex[(b^1_n)^2]]\ex[({\tilde Y}^n(\infty))^2]-\g_n \delta\ex[({\tilde X}^n(\infty))]\\ &+\frac12[ \delta^2\gamma +(\delta+ \delta\kappa n^\al)\ex[b^2_n({\tilde X}^n(\infty), {\tilde Y}^n(\infty))]]
\end{align*}
\end{lem}
\begin{proof}[Proof of Lemma \ref{lem:2nd_moment_b_2}]
Let $u(x,y)=xy$, we have, from Lemma   \ref{lem:vanishing},
\begin{align*}
&\ex \left[\g _n \delta {\tilde Y}^n(\infty) - \frac{1}{\delta} \left[\left({\tilde X}^n(\infty)+n\delta\right)\wedge \left(\frac{{\tilde Y}^n(\infty)\delta}{\eta}+ (n+\kappa n^\al)\delta\right)\right] \delta {\tilde Y}^n(\infty)\right. \\ &+ \left. \gamma \delta {\tilde X}^n(\infty) - \frac{1}{\delta}\left[\left(\frac{{\tilde Y}^n(\infty)\delta}{\eta}-{\tilde X}^n(\infty)+\delta\kappa n^\al \right)^+\right] \delta {\tilde X}^n(\infty) \right]=0.
\end{align*}
Therefore, 
\begin{align*}
	&\ex\left[ \frac{1}{\delta}\left[\left(\frac{{\tilde Y}^n(\infty)\delta}{\eta}-{\tilde X}^n(\infty)+\delta\kappa n^\al \right)^+\right] \delta {\tilde X}^n(\infty) \right]\\= &
	 \ex \left[\g_n  \delta {\tilde Y}^n(\infty) - \frac{1}{\delta} \left[\left({\tilde X}^n(\infty)+n\delta\right)\wedge \left(\frac{{\tilde Y}^n(\infty)\delta}{\eta}+ (n+\kappa n^\al)\delta\right)\right] \delta {\tilde Y}^n(\infty)-\gamma \delta {\tilde X}^n(\infty)\right]
	 \\ \ge & \ex \left[\g_n  \delta {\tilde Y}^n(\infty) -\gamma \delta {\tilde X}^n(\infty)\right]- \ex[b_1^2] \ex[({\tilde Y}^n(\infty))^2]
\end{align*}
where the inequality is due to the Cauchy-Schwartz inequality.  Meanwhile, $u(x,y)=y^2$ leads to,
\begin{align*}
	\ex\left[\gamma [2\delta {\tilde Y}^n(\infty)+\delta^2] - \frac{1}{\delta}\left[\left(\frac{{\tilde Y}^n(\infty)\delta}{\eta}-{\tilde X}^n(\infty)+\delta\kappa n^\al \right)^+\right][2\delta {\tilde Y}^n(\infty) -\th^2]\right]=0.
\end{align*}
Hence, 
\begin{align*}
	\ex[\gamma [2\delta {\tilde Y}^n(\infty)+\delta^2] ]&=\ex\left[\frac{1}{\delta}\left[\left(\frac{{\tilde Y}^n(\infty)\delta}{\eta}-{\tilde X}^n(\infty)+\delta\kappa n^\al \right)^+\right][2\delta {\tilde Y}^n(\infty) -\th^2]\right]
\end{align*}
Plug it into the previous one, we have, 
\begin{align*}
	&\ex[\gamma [2\delta {\tilde Y}^n(\infty)+\delta^2] ]+2\ex \left[\g \delta {\tilde Y}^n(\infty) -\gamma \delta {\tilde X}^n(\infty)\right]- 2\ex[b_1^2] \ex[({\tilde Y}^n(\infty))^2]\\ \ge &\ex\left[\frac{1}{\delta}\left[\left(\frac{{\tilde Y}^n(\infty)\delta}{\eta}-{\tilde X}^n(\infty)+\delta\kappa n^\al \right)^+\right][2\delta {\tilde Y}^n(\infty)- 2\delta {\tilde X}^n(\infty)-\delta^2]\right]
	\\ = &  \ex\left[\frac{1}{\delta}\left[\left(\frac{{\tilde Y}^n(\infty)\delta}{\eta}-{\tilde X}^n(\infty)+\delta\kappa n^\al \right)^+\right][2\delta \frac{{\tilde Y}^n(\infty)\delta}{\eta}- 2\delta {\tilde X}^n(\infty)+\delta^2\kappa n^\al]\right]\\& -(\delta+ \delta\kappa n^\al)\ex\left({\tilde Y}^n(\infty)-{\tilde X}^n(\infty)+\delta\kappa n^\al \right)^+
	\\=& 2\ex[b_2^2({\tilde X}^n(\infty), {\tilde Y}^n(\infty))]-(\delta+ \delta\kappa n^\al)\ex[b_2({\tilde X}^n(\infty), {\tilde Y}^n(\infty))]
\end{align*}
Therefore, we have,
\begin{align*}
\ex[b_2^2({\tilde X}^n(\infty), {\tilde Y}^n(\infty))] %\le & \frac12 \left\{ \ex[\gamma [2\delta {\tilde Y}^n(\infty)+\delta^2] ]+2\ex \left[\g \delta {\tilde Y}^n(\infty) -\gamma \delta {\tilde X}^n(\infty)\right]- 2\ex[b_1^2] \ex[({\tilde Y}^n(\infty))^2]+(\delta+ \delta\kappa n^\al)\ex[b_2({\tilde X}^n(\infty), {\tilde Y}^n(\infty))]\right\}
\le &  (\delta \gamma + \g_n \delta)\ex[({\tilde Y}^n(\infty))] - \ex[b_1^2]\ex[({\tilde Y}^n(\infty))^2]-\g_n \delta\ex[({\tilde X}^n(\infty))]\\ &+\frac12[ \delta^2\gamma +(\delta+ \delta\kappa n^\al)\ex[b_2({\tilde X}^n(\infty), {\tilde Y}^n(\infty))]].
\end{align*}
\end{proof}

\subsubsection{Derivative Bounds}
\label{sec:gradient_bound}

Recall that we need to bound terms related to the derivatives of solution to the Stein equation $\calg u(x,y) = H(x,y)$, with  $H(x,y) = h(x, y) -\ex h(X, Y)$, and 
\begin{align*}
\calg u(x, y) = & \delta[\g_n -b_1(x,y)\mu]\frac{\partial}{\partial x}u(x,y) + \th[ \gamma -b_2(x,y)\nu] \frac{\partial}{\partial x}u(x,y)  +\frac{1}{2}\frac{\partial^2}{\partial x^2}  u(x,y).
\end{align*}
Note that the second order derivative is only related to the $x$ direction, which reflects the fact that the randomness in the two-dimensional diffusion process comes from a one dimensional Brownian motion. 
This type of equation belongs to the family of degenerated Kolmogorov equations, for background, and detailed analysis, see, e.g. \cite{MENOZZI2018756, Talay2002StochasticHS, soizebook}. Furthermore, the special form of the differential equation in our system allows us to further reduce it to an ordinary differential equation(ODE). More specifically, note that the Stein equation bears the following form, 
\begin{align}
\label{eqn:pde+original}
\delta[\g_n -b_1(x,y)\mu]u_x(x,y) + \eta[ \gamma -b_2(x,y)\nu] u_y(x,y)  +\frac{1}{2}u_{xx}(x,y)= H(x,y).
\end{align}
Consider two separate domains. On $\{ \frac{x}{\delta}+ n \ge \frac{y}{\eta} +(n+\kappa n^\al)\}$ %$\left(x+n\delta\right) \ge y+ (n+\kappa n^\al)\delta$, 
\eqref{eqn:pde+original} becomes, 
\begin{align}
\label{eqn:pde+original_I}
\delta\left[\g_n -\left(\frac{y}{\eta} +(n+\kappa n^\al)\right)\mu\right]u_x(x,y) + \eta \gamma u_y(x,y)+\frac{1}{2}u_{xx}(x,y)= H(x,y).
\end{align}
or equivalently, 
\begin{align}
-\left[\frac{y\delta}{\eta}\mu + \kappa n^\al\delta\mu\right]u_x(x,y) +\eta \gamma u_y(x,y) +\frac{1}{2}u_{xx}(x,y)= H(x,y).
\end{align}
When $\frac{x}{\delta}+ n<\frac{y}{\eta} +(n+\kappa n^\al)$, 
%$\left(x+n\delta\right) < \left(y+ (n+\kappa n^\al)\delta\right)$, 
\eqref{eqn:pde+original} takes the form, 
\begin{align}
&\delta\left[\g_n -\frac{1}{\delta}(x+n\delta)\mu\right]u_x(x,y) \nonumber \\ &+ \delta\left[ \gamma -\left(\frac{y\delta}{\eta}-x+\kappa n^\al\right)\nu\right] u_y(x,y)  +\frac{1}{2}u_{xx}(x,y)= H(x,y).\label{eqn:pde+original_II}
\end{align}
or equivalently,
\begin{align}
\label{eqn:pde+original_II}
-x\mu u_x(x,y) -\left(\frac{y\delta}{\eta}-x\right)\nu u_y(x,y)  +\frac{1}{2}u_{xx}(x,y)= H(x,y).
\end{align}

\subsubsection{The solution in Domain I}

In Domain I: $\{ (x, y): \frac{x}{\delta}+ n \ge \frac{y}{\eta} +(n+\kappa n^\al)\}$, we have the equation \eqref{eqn:pde+original_I}.
%\begin{align*}
%\delta\left[\g-\frac{1}{\delta}(y+ (n+\kappa n^\al)\delta)\mu\right]u_x(x,y) +\eta \gamma u_y(x,y) +\frac{1}{2}u_{xx}(x,y)= H(x,y).
%\end{align*}
From well-known results on linear elliptic equation, see e.g. \cite{krylov1996lectures}, we know that the solution exists, and its Sobolev norm of $u$ is bounded by that of the $H(x,y)$ and the boundary condition, that is, $u_x$ and $u_y$ are bounded in $L_p$ space for a proper $p$. The solution can also be observed to have the following presentation,  
\begin{align*}
-\left[\frac{y\delta}{\eta}\mu + \kappa n^\al\delta\mu\right]u_x(x,y)  +\frac{1}{2}u_{xx}(x,y) &= H(x,y) + Q(x,y)\\  \eta \gamma u_y &= -Q(x,y).
\end{align*}
for some function $Q(x,y)$. Thus, 
\begin{align*}
u(x,y)= & \int_0^x \left[\int_0^w \exp\left(-\int_0^z 2(y+\kappa n^\al )\mu du\right) [H(x,y) + Q(x,y)]dz\right]\\ &\cdot \left[ \exp\left(\int_0^w 2(y+\kappa n^\al )\mu du \right)\right]dw.
\end{align*}
Since the solution to a linear second order differential equation, 
\begin{align*}
f''(x) + a(x) f'(x) + b(x)=0,
\end{align*}
with proper boundary condition will have a solution in the form of 
\begin{align*}
f(x)= \int_0^x \left[\int_0^w \exp\left(\int_0^za(u) du\right) b(z) dz\right]\left[ \exp\left(-\int_0^w a(u) du \right)\right]dw.
\end{align*}
and 
\begin{align*}
f'(x) = \left[\int_0^x \exp\left(\int_0^za(u) du\right) b(z) dz\right]\left[ \exp\left(-\int_0^x a(u) du \right)\right].
\end{align*}

Direct calculations, similar to those in \cite{doi:10.1287/moor.2013.0593, gurvich2014, braverman2017},  thus provides us with the following bound for the solutions in domain I. 
\begin{lem}
\label{lem:gradient_bound_in_domain_I}
$u_{xxx}$ and $u_{yy}$ are bounded quantities in domain I.
\end{lem}

\subsubsection{The solution in Domain II}

In Domain II: $\{ (x, y):\frac{x}{\delta}+ n < \frac{y}{\eta} +(n+\kappa n^\al)\}$, we have the equation \eqref{eqn:pde+original_II}. 
%
%\begin{align*}
%&\delta\left[\g_n -\frac{1}{\delta}(x+n\delta)\mu\right]u_x(x,y) \\& + \delta\left[ \gamma -\frac{1}{\theta}(y-x+\theta\kappa n^\al)\nu\right] u_y(x,y)  +\frac{1}{2}u_{xx}(x,y)= H(x,y).
%\end{align*}
%Or
%\begin{align}
%	\label{eqn:domain_II}
%-xu_x(x,y) - (x-y) u_y + \frac{1}{2}u_{xx}(x,y)= H(x,y).
%\end{align}
The solution in domain I provide the values of $u(x,y)$ on the line $\{ (x, y):\frac{x}{\delta}+ n = \frac{y}{\eta} +(n+\kappa n^\al)\}$, this serves as part of the boundary conditions for the solution in domain II, the other part is, of course, the original boundary condition. We will provide the necessary estimation via {\it a priori} estimation of its solution, that is, obtain those estimation without solving the equation. 
\eqref{eqn:pde+original_II} implies that, for any smooth function  $\phi(x,y)$ of polynomial growth, since $H(x,y)$ is assumed to be integrable, 
\begin{align*}
&\int_{-\infty}^\infty \int_{-\infty}^{y+\kappa} [-x u_x- (y-x) u_y +\frac12 u_{xx} ) \phi(x,y) dx dy\\=&\int_{-\infty}^\infty \int_{-\infty}^{y+\kappa} \phi(x,y) H(x,y) dx dy. 
\end{align*}
Note that, for the ease of exposition, we only discuss the case $\eta=\delta$. It is easy to see that the results extend to general case.  
Integration by part gives us, 
\begin{align}
-&\int_{-\infty}^\infty \int_0^{y+\kappa}  2u(x,y) \phi(x,y) + u(x,y) [x\phi_x + (y-x) \phi_y + \frac12\phi_{xx} ] dx dy \nonumber \\= & \int_{-\infty}^\infty y\phi(y,y) u(y,y) dy+\int_{-\infty}^\infty \int_{-\infty}^{y+\kappa} \phi(x,y) H(x,y) dx dy. \label{eqn:ibp_first}
\end{align}

\begin{lem}
\label{lem:Estimate_unified}
For any bounded set $\Omega$ in Domain II, there exists a constant $C_1$, such that, 
\begin{align*}
\int \int_\Omega | u(x,y) | dxdy \le C_1.
\end{align*}
\end{lem}
\begin{proof}
Let $\phi(x,y) = \eta u$ with $\eta$ being a smooth function with suitable growth and $u$ being the solution. The existence and integrability of itself and its generalized derivatives (regularity in Sobolev spaces) have been established in \cite{MENOZZI2018756, Talay2002StochasticHS}. Thus, 
\begin{align*}
&x\phi_x + (y-x) \phi_y +\phi_{xx} \\= &x\eta_x u + x\eta u_x + (y-x) \eta_y u + (y-x) \eta u_y +\eta_{xx} u +2\eta_x u_x +\eta u_{xx}\\ =&[x\eta_x+ (y-x) \eta_y +\eta_{xx} ]u + 2\eta_x u_x +\eta[x u_x  + (y-x)  u_y +u_{xx}]\\ =&[x\eta_x+ (y-x) \eta_y +\eta_{xx}] u +2\eta_x u_x+\eta[x u_x  (y-x)  u_y- u_{xx}]+ 2\eta u_{xx}.
\end{align*}
Plug it into \eqref{eqn:ibp_first}. This follows the same approach that is conducted in \cite{bensoussan2013regularity}, we can have a cut-off and/or mollifier of $u$ instead of $u$ itself if necessary. 
Apparently, the term $\int x\eta_x+ (y-x) \eta_y u-\eta_{xx} u$ is finite and known. Of course, the above quantity  equal to $\int_{-\infty}^\infty y\phi(y,y) u(y,y) dy$, which is known. 

The term, 
\begin{align*}
\int_{-\infty}^\infty \int_{-\infty}^{y+\kappa}\eta_xu_x dx dy &= \int_0^\infty \int_0^y u_x d\eta dy\\ &= \int_{-\infty}^\infty u_x(y,y) \eta(y,y) - \int_{-\infty}^\infty \int_{-\infty}^{y+\kappa}\eta u_{xx} dx dy.
\end{align*}
This will cancel the term $\eta u_{xx}$, hence, we can conclude that there exists a $C$ such that, 
\begin{align*}
\int_{-\infty}^\infty \int_{-\infty}^{y+\kappa} [x\eta_x+ (y-x) \eta_y +\eta_{xx}] u dx dy \le C
\end{align*}
Now we can pick a proper $\eta$ to have the desired result. For example, $\eta$ is taken as $\frac{x^2}{2} {\bf 1}\Omega$, then we can conclude $\int_\Omega |x+1|^2 u \le C$, which is suffient for the desired result. 
\end{proof}

\begin{lem}
\label{lem:Estimate_unified_II}For any bounded set $\Omega$ in Domain II, there exist positive constants $C_2$ and $C_3$, such that, 
\begin{align*}
\int \int_\Omega \vert u_{xxx} (x,y) \vert   dx dy \le C_2, \quad \int \int_\Omega \vert  u_{yy} (x,y) \vert   dx dy \le C_3.
\end{align*}
\end{lem}
\begin{proof}
The above arguments also applies to $\phi=\eta u_{xxx}$ and $\phi=\eta u_{yy}$. In fact, this type of estimation falls into the general category of the Bernstein techniques, see, e.g. \cite{OleKru61}. Here, we made use of the solution in domain I, and some explicit calculation in the place of maximum principle that is normally instrumental in applying Bernstein techniques.
\end{proof}
Lemmas \ref{lem:Estimate_unified} and \ref{lem:Estimate_unified_II}, in conjunction with one of the moment bounds, implies that, 
%\begin{cor}
%\label{cor:unify} $\ex[E_i]=O(1/\sqrt{n})$ for $i=1, 3, 4$ and $\ex[E_2]= %O(1/n)$. Meanwhile, we have the more sophisticated estimates,
%\begin{align*}
%\ex[E_i] &= O(n\delta^3),i=1,2, \quad  \ex[E_i] = O(n^\al \delta^2), i=3,4 
%\end{align*}
%\end{cor}
\begin{proof}[Proof of Lemma \ref{lem:error}]
To show that $\ex[E_1]= O(n\delta^3)$, we only need that $u_{xxx}$ is locally integrable, which is the result of Lemmas \ref{lem:gradient_bound_in_domain_I} and  \ref{lem:Estimate_unified_II}. The same lemmas, in conjunction with Lemma \ref{lem:b1_mean}, will guarantee that $\ex[E_2]=  O(n\delta^3)$. Then, Lemma \ref{lem:gradient_bound_in_domain_I} for domain I and Lemma \ref{lem:Estimate_unified_II} for domain II  indicate the boundedness of $u_{yy}$, together with Lemmas \ref{lem:b2_mean}, they imply, 
$\ex[E_i]= O(n^a\eta^2)$ for $i=3,4$.
\end{proof}

\bibliography{Lu}{}

\begin{thebibliography}{15}
\providecommand{\natexlab}[1]{#1}
\providecommand{\url}[1]{{#1}}
\providecommand{\urlprefix}{URL }
\providecommand{\doi}[1]{\url{https://doi.org/#1}}
\providecommand{\eprint}[2][]{\url{#2}}
 \bibcommenthead

\bibitem[{Bensoussan and Frehse(2013)}]{bensoussan2013regularity}
Bensoussan A, Frehse J (2013) Regularity Results for Nonlinear Elliptic Systems
  and Applications. Applied Mathematical Sciences, Springer Berlin Heidelberg,
  \urlprefix\url{https://books.google.com/books?id=tq77CAAAQBAJ}

\bibitem[{Bhandari et~al(2008)Bhandari, Scheller-Wolf, and
  Harchol-Balter}]{BSH2008}
Bhandari A, Scheller-Wolf A, Harchol-Balter M (2008) An exact and efficient
  algorithm for the constrained dynamic operator staffing problem for call
  centers. Management Science 54(2):339--353

\bibitem[{Braverman and Dai(2017)}]{braverman2017}
Braverman A, Dai JG (2017) Stein's method for steady-state diffusion
  approximations of $m/\mathit{Ph}/n+m$ systems. Ann Appl Probab
  27(1):550--581. \doi{10.1214/16-AAP1211},
  \urlprefix\url{https://doi.org/10.1214/16-AAP1211}

\bibitem[{Fayolle et~al(1999)Fayolle, Iasnogorodski, Malyshev, and
  Malyshev}]{fayolle1999random}
Fayolle G, Iasnogorodski R, Malyshev V, et~al (1999) Random Walks in the
  Quarter-Plane: Algebraic Methods, Boundary Value Problems and Applications.
  Applications of mathematics, Springer,
  \urlprefix\url{https://books.google.com/books?id=Uuyw1Jdh0xgC}

\bibitem[{Gurvich(2014{\natexlab{a}})}]{gurvich2014}
Gurvich I (2014{\natexlab{a}}) Diffusion models and steady-state approximations
  for exponentially ergodic markovian queues. Ann Appl Probab 24(6):2527--2559.
  \doi{10.1214/13-AAP984}, \urlprefix\url{https://doi.org/10.1214/13-AAP984}

\bibitem[{Gurvich(2014{\natexlab{b}})}]{doi:10.1287/moor.2013.0593}
Gurvich I (2014{\natexlab{b}}) Validity of heavy-traffic steady-state
  approximations in multiclass queueing networks: The case of queue-ratio
  disciplines. Mathematics of Operations Research 39(1):121--162.
  \doi{10.1287/moor.2013.0593},
  \urlprefix\url{https://doi.org/10.1287/moor.2013.0593},
  {\href{https://arxiv.org/abs/https://doi.org/10.1287/moor.2013.0593}{{https://arxiv.org/abs/https://doi.org/10.1287/moor.2013.0593}}}

\bibitem[{Halfin and Whitt(1981)}]{RePEc:inm:oropre:v:29:y:1981:i:3:p:567-588}
Halfin S, Whitt W (1981) {Heavy-Traffic Limits for Queues with Many Exponential
  Servers}. Operations Research 29(3):567--588. \doi{10.1287/opre.29.3.567},
  \urlprefix\url{https://ideas.repec.org/a/inm/oropre/v29y1981i3p567-588.html}

\bibitem[{Krylov(1996)}]{krylov1996lectures}
Krylov N (1996) Lectures on Elliptic and Parabolic Equations in Holder Spaces.
  Graduate studies in mathematics, American Mathematical Society,
  \urlprefix\url{https://books.google.com/books?id=oh4SCgAAQBAJ}

\bibitem[{Mazalov and Gurtov(2012)}]{MazalovGurtov2012}
Mazalov V, Gurtov A (2012) Queueing system with on-demand number of servers.
  Mathematica Applicanda 40(2):1--12

\bibitem[{Menozzi(2018)}]{MENOZZI2018756}
Menozzi S (2018) Martingale problems for some degenerate kolmogorov equations.
  Stochastic Processes and their Applications 128(3):756--802.
  \doi{https://doi.org/10.1016/j.spa.2017.06.001},
  \urlprefix\url{https://www.sciencedirect.com/science/article/pii/S0304414917301564}

\bibitem[{Meyn and Tweedie(1993)}]{meyn_tweedie_1993}
Meyn SP, Tweedie RL (1993) Stability of markovian processes iii:
  Foster–lyapunov criteria for continuous-time processes. Advances in Applied
  Probability 25(3):518–548. \doi{10.2307/1427522}

\bibitem[{Oleinik and Kruzhkov(1961)}]{OleKru61}
Oleinik OA, Kruzhkov SN (1961) Quasi-linear second-order parabolic equations
  with many independent variables. Uspekhi Mat Nauk 16(5):115--155

\bibitem[{Soize(1994)}]{soizebook}
Soize C (1994) The Fokker-Planck equation for stochastic dynamical systems and
  its explicit steady state solutions. World Scientific Publishing Co., Inc.,
  River Edge, NJ,

\bibitem[{Stein(1986)}]{stein86}
Stein C (1986) Approximate computation of expectations. 7, Institute of
  Mathematical Statistics Lecture Notes, Monograph Series

\bibitem[{Talay(2002)}]{Talay2002StochasticHS}
Talay D (2002) Stochastic hamiltonian systems : Exponential convergence to the
  invariant measure , and discretization by the implicit euler scheme. Markov
  Processes and Related Fields 8(2):163--198

\end{thebibliography}
\end{document}